\documentclass[12pt, reqno]{amsart}
\usepackage{amsmath, amsthm, amscd, amsfonts, amssymb, graphicx, color}
\usepackage[bookmarksnumbered, colorlinks, plainpages]{hyperref}
\hypersetup{colorlinks=true,linkcolor=red, anchorcolor=green, citecolor=cyan, urlcolor=red, filecolor=magenta, pdftoolbar=true}

\newtheorem{theorem}{Theorem}[section]
\newtheorem{lemma}[theorem]{Lemma}

\newtheorem{corollary}[theorem]{Corollary}
\theoremstyle{definition}

\theoremstyle{remark}
\newtheorem{remark}[theorem]{Remark}
\numberwithin{equation}{section}

\begin{document}
\setcounter{page}{1}

\title{On statistical approximation properties of $(p,q)$-Bleimann-Butzer-Hahn operators}
\author[M. Mursaleen and Taqseer Khan ]{M. Mursaleen$^1$ and Taqseer Khan$^2$}

\address{$^{1}$ Department of Mathematics, Aligarh Muslim University, Aligarh--202002, India.}
\email{\textcolor[rgb]{0.00,0.00,0.84}{mursaleenm@gmail.com}}

\address{$^{2}$ Department of Pure Mathematics, Aligarh Muslim University, Aligarh--202002, India.}
\email{\textcolor[rgb]{0.00,0.00,0.84}{taqi.khan91@gmail.com}}





\begin{abstract}
The aim of this paper is to introduce a generalization of the $(p,q)$-Bleimann-Butzer-Hahn operators based on $(p,q)$-integers and obtain
Korovkin's type statistical approximation theorem for these operators. Also, we
establish the rate of convergence of these operators using the modulus of continuity. Furthermore, we introduce $(p,q)$-Bleimann-Butzer-Hahn bivariate operators.
\end{abstract} \maketitle
\section{Introduction and Preliminaries}
In order to approximate continuous functions defined on the positive half axis, Bleimann, Butzer and Hahn (BBH) introduced, in 1980, the following linear positive operators in \cite%
{brns};
\begin{align}\label{nas1}
L_n (f;x)= \frac{1}{(1+x)^n}\sum_{k=0}^n f\left(\frac{k}{n-k+1} \right)\left[
\begin{array}{c}
n \\
k%
\end{array}%
\right] x^k, x\geq 0
\end{align}
\parindent=8mmThe advent of $q$-calculus created a new venue of research in approximation theory. Lupa\c{s} \cite{lups} introduced the first $q$-analogue of the Bernstein polynomials in 1987. Phillips \cite{philip} presented another modification of Bernstein polynomials in 1997. He also established results for the convergence and the Voronovskaja's type asymptotic
expansion for these operators.

The $q$-analogue of the BBH-type operators is defined as
\begin{align}\label{nas2}
L_{n}^{q}(f;x)=\frac{1}{\ell _{n}(x)}\sum_{k=0}^{n}f\left( \frac{[k]_{q}}{%
[n-k+1]_{q}q^{k}}\right) q^{\frac{k(k-1)}{2}}\left[
\begin{array}{c}
n \\
k%
\end{array}%
\right] _{q}x^{k}
\end{align}
where $\ell _{n}(x)=\prod_{k=0}^{n-1}(1+q^{s}x)$.\newline
In recent decades, the concept of $(p,q)$-calculus has also been introduced. Many researchers have used $(p,q)$-calculus to establish new and interesting results in approximation theory. Recently, Mursaleen et al \cite{mur7} introduced the first $(p,q)$-analogue of Bernstein operators and $(p,q)$-analogue of Bernstein-Stancu operators \cite{mur8}. They have investigated the approximation properties and convergence properties of these operators.
\newline

Let us give rudiments of $(p,q)$-calculus.

The $(p,q)$ integers $[n]_{p,q}$ are defined by
\begin{equation*}
[n]_{p,q}=\frac{p^n-q^n}{p-q},~~~~~~~n=0,1,2,\cdots, ~~0<q<p\leq 1.
\end{equation*}
whereas $q$-integers are given by
\begin{equation*}
[n]_{q}=\frac{1-q^n}{1-q},~~~~~~~n=0,1,2,\cdots, ~~0<q< 1.
\end{equation*}%
\newline
It is very clear that the two concepts are different but the former is a generalization of the later.\newline

Also, we have $(p,q)$%
-binomial expansion as follows
\begin{equation*}
(ax+by)_{p,q}^{n}:=\sum\limits_{k=0}^{n}p^{\frac{(n-k)(n-k-1)}{2}}q^{\frac{%
k(k-1)}{2}} \left[
\begin{array}{c}
n \\
k%
\end{array}%
\right] _{p,q}a^{n-k}b^{k}x^{n-k}y^{k},
\end{equation*}
\begin{equation*}
(x+y)_{p,q}^{n}=(x+y)(px+qy)(p^2x+q^2y)\cdots (p^{n-1}x+q^{n-1}y),
\end{equation*}
\begin{equation*}
(1-x)_{p,q}^{n}=(1-x)(p-qx)(p^2-q^2x)\cdots (p^{n-1}-q^{n-1}x)
\end{equation*}%
\newline
and the $(p,q)$-binomial coefficients are defined by
\begin{equation*}
\left[
\begin{array}{c}
n \\
k%
\end{array}%
\right] _{p,q}=\frac{[n]_{p,q}!}{[k]_{p,q}![n-k]_{p,q}!}.
\end{equation*}

By some simple calculation, we have the following relation

\begin{equation*}
q^k[n-k+1]_{p,q}=[n+1]_{p,q}-p^{n-k+1}[k]_{p,q}.
\end{equation*}

For details on $q$-calculus and $(p,q)$-calculus, one is referred to \cite{vp} and
\cite{mah,sad,vivek} respectively.\\

The concept of statistical convergence was introduced by Fast \cite{fast} in the circa 1950 and in recent times it has become an active area of research. The concept of the limit of a sequence has been generalized to a statistical limit through the natural density of a set $K$ of positive integers, defined as
\begin{equation*}
\delta(K) = \lim_{n \to \infty}\frac{1}{n} \{k\leqslant n ~for ~k\in K\}
\end{equation*}
provided this limit exists \cite{niv}. We say that the sequence $x= (x_n)$ statistically converges to a number $l$, if for each $\varepsilon>0$, the density of the set $\{k : |x_k-l|\geqslant\varepsilon\}$ is zero. We denote it by $st-\lim_k x_x = l$. It is easily seen that every convergent sequence is statistically convergent but not inversely.\\

The main purpose of this paper is to introduce a modification of the operators defined by Mursaleen et al. \cite{mur18} and investigate statistical approximation properties of the operators with the aid of Korovkin type theorem and estimate the rate of their statistical convergence.

Now based on $(p,q)$-integers, we construct $(p,q)$-analogue of
BBH operators, and we call them as $(p,q)$-Bleimann-Butzer-Hahn Operators and
investigate their Korovokin's type statistical approximation properties by using the test
functions $\left(\frac{t}{1+t}\right)^\nu$ for $\nu=0,1,2$. Also for a space
of generalized Lipschitz-type maximal functions we give a pointwise
estimation.

Let $C_{B}(\mathbb{R}_+)$ be the set of all bounded and continuous functions
on $\mathbb{R}_+$, then $C_{B}(\mathbb{R}_+)$ is linear normed space with
\begin{equation*}
\parallel f \parallel_{C_{B}}= \sup_{x \geq 0} \mid f(x) \mid.
\end{equation*}
Let $\omega$ denotes modulus of continuity satisfying the following
condition:
\begin{enumerate}
\item $\omega$ is a non-negative increasing function on $\mathbb{R}%
_+$

\item $\omega(\delta_1+\delta_2)\leq
\omega(\delta_1)+\omega(\delta_2)$

\item $\lim_{\delta \to 0}\omega(\delta)=0$.
\end{enumerate}
Let ${H}_\omega$ be the space of all real-valued functions $f$ defined on
the semiaxis $\mathbb{R}_+$ satisfying the condition
\begin{equation*}
\mid f(x) - f(y) \mid \leq \omega \left(\bigg{|} \frac{x}{1+x}- \frac{y}{1+y}
\bigg{|} \right),
\end{equation*}
for any $x,y \in \mathbb{R}_+$.
\begin{theorem}\label{main}\protect\cite{butz}
Let $\{A_n\}$ be the sequence of positive linear operators from $H_\omega$
into $C_B(\mathbb{R}_+)$, satisfying the conditions
\begin{equation*}
\lim_{n \to \infty} \parallel A_n \left( \left( \frac{t}{1+t}%
\right)^\nu;x\right)-\left(\frac{x}{1+x}\right)^\nu \parallel_{C_{B}},
\end{equation*}
for $\nu=0,1,2$. Then for any function $f \in H_\omega$
\begin{equation*}
\lim_{n \to \infty} \parallel A_n (f)-f \parallel_{C_{B}}=0.
\end{equation*}
\end{theorem}
Now we introduce $(p,q)$-Bleimann-Butzer-Hahn type operators based on $(p,q)$%
-integers as follows:
\begin{align}\label{nas3}
L_n^{p,q}(f;x)=\frac{pq}{\ell_n^{p,q}(x)}\sum_{k=0}^n f \left( \frac{
p^{n-k+1}[k]_{p,q}}{[n-k+1]_{p,q}q^k }\right) p^{\frac{(n-k)(n-k-1)}{2}}q^{%
\frac{k(k-1)}{2}} \left[
\begin{array}{c}
n \\
k%
\end{array}%
\right] _{p,q} x^k
\end{align}
where, ~~$x \geq 0,~~0< q<p\leq 1$
\begin{equation*}
\ell_n^{p,q}(x)= \prod_{s=0}^{n-1}(p^s+q^s x)
\end{equation*}
and $f$ is defined on semiaxis $\mathbb{R}_+$.\newline
And also by induction, we construct the Euler identity based on $(p,q)$%
-analogue defined as follows:
\begin{align}\label{nas4}
\prod_{s=0}^{n-1}(p^s+q^s x) =\sum_{k=0}^n p^{\frac{(n-k)(n-k-1)}{2}}q^{%
\frac{k(k-1)}{2}} \left[
\begin{array}{c}
n \\
k%
\end{array}%
\right] _{p,q} x^k
\end{align}

\parindent=8mmIf we put $p=1$, then we obtain $q$-BBH operators. In \eqref{nas3},
if we take $f\left( \frac{[k]_{p,q}}{[n-k+1]_{p,q}}\right) $ in place of $%
f\left( \frac{p^{n-k+1}[k]_{p,q}}{[n-k+1]_{p,q}q^{k}}\right) $ ,
then we obtain the usual generalization of Bleimann, Butzer and Hahn operators
based on $(p,q)$-integers and  then it is not possible to obtain
explicit expressions for the monomials $t^{\nu }$ and $\left( \frac{t}{1+t}%
\right) ^{\nu }$ for $\nu =1,2$. Explicit formulas for the
monomials $\left( \frac{t}{1+t}\right) ^{\nu }$ for $\nu =0,1,2$ are obtainable only if we define the Bleimann, Butzer and
Hahn operators as in \eqref{nas3}. It is to note that these operators are more flexible than the classical
BBH operators and $q$-analogue of BBH operators. That is depending on the
selection of $(p,q)$-integers, the rate of convergence of $(p,q)$%
-BBH operators is  as good as the classical one atleast.

\section{Main results}

\begin{lemma}\label{main1}
Let $L_n^{p,q}(f;x)$ be given by \eqref{nas3}, then for any $x \geq 0$ and $%
0<q<p\leq 1$ we have the following identities

\begin{enumerate}
\item $L_n^{p,q}(1;x)=pq,$

\item\label{main2} $L_n^{p,q}(\frac{t}{1+t};x)=\frac{p^2q[n]_{p,q}}{[n+1]_{p,q}}%
\left(\frac{x}{1+x}\right),$

\item \label{main4}$L_n^{p,q}\left((\frac{t}{1+t})^2;x\right)=\frac{%
p^2q^3[n]_{p,q}[n-1]_{p,q}}{[n+1]_{p,q}^2}\frac{x^2}{(1+x)(p+qx)} +\frac{%
p^{n+2}q[n]_{p,q}}{[n+1]_{p,q}^2}\left(\frac{x}{1+x}\right).$
\end{enumerate}
\end{lemma}

\begin{proof}
\begin{enumerate}
\item $%
L_{n}^{p,q}\left( 1;x\right) =\frac{pq}{\ell _{n}^{p,q}(x)}\sum_{k=1}^{n}%
 p^{\frac{(n-k)(n-k-1)}{2}}q^{\frac{k(k-1)}{2}}\left[
\begin{array}{c}
n \\
k%
\end{array}%
\right] _{p,q}x^{k}$ \\ \newline

but for $0<q<p\leq 1$, we have
\begin{equation*}
\sum_{k=0}^{n}p^{\frac{(n-k)(n-k-1)}{2}}q^{\frac{k(k-1)}{2}}\left[
\begin{array}{c}
n \\
k%
\end{array}%
\right] _{p,q}x^{k}=\prod_{s=0}^{n-1}(p^{s}+q^{s}x)=\ell_{n}^{p,q}(x),
\end{equation*}
so
\begin{equation*}
L_n^{p,q}(1;x)=\frac{pq}{l_n^{p,q}(x)}\times l_n^{p,q}(x)=pq.
\end{equation*}
This proves (1).\\ \newline

\item Let $t=\frac{p^{n-k+1}[k]_{p,q}}{[n-k+1]_{p,q}q^{k}}$, then $%
\frac{t}{t+1}=\frac{[k]_{p,q}p^{n+1-k}}{[n+1]_{p,q}}$
\begin{eqnarray*}
L_{n}^{p,q}\left( \frac{t}{1+t};x\right)  &=&\frac{pq}{\ell _{n}^{p,q}(x)}%
\sum_{k=1}^{n}\frac{[k]_{p,q}p^{n-k+1}}{[n+1]_{p,q}}p^{\frac{(n-k)(n-k-1)}{2}%
}q^{\frac{k(k-1)}{2}}\left[
\begin{array}{c}
n \\
k%
\end{array}%
\right] _{p,q}x^{k} \\
&=&\frac{pq}{\ell _{n}^{p,q}(x)}\sum_{k=1}^{n}\frac{[n]_{p,q}p^{n-k+1}}{%
[n+1]_{p,q}}p^{\frac{(n-k)(n-k-1)}{2}}q^{\frac{k(k-1)}{2}}\left[
\begin{array}{c}
n-1 \\
k-1%
\end{array}%
\right] _{p,q}x^{k} \\
&=&x\left( \frac{pq}{\ell _{n}^{p,q}(x)}\cdot \frac{\lbrack n]_{p,q}}{%
[n+1]_{p,q}}p\right) \sum_{k=0}^{n-1}p^{\frac{(n-k)(n-k-1)}{2}}q^{\frac{%
k(k-1)}{2}}\left[
\begin{array}{c}
n-1 \\
k%
\end{array}%
\right] _{p,q}(qx)^{k} \\
&=&p^2q\frac{[n]_{p,q}}{[n+1]_{p,q}}\frac{x}{1+x}.
\end{eqnarray*}
This completes the proof of (2).
\item $L_n^{p,q}\left(\frac{t^2}{(1+t)^2};x\right) =\frac{pq}{%
\ell_n^{p,q}(x)}\sum_{k=1}^n \frac{[k]_{p,q}^2p^{2(n-k+1)}}{[n+1]_{p,q}^2}
p^{\frac{(n-k)(n-k-1)}{2}}q^{\frac{k(k-1)}{2}} \left[
\begin{array}{c}
n \\
k%
\end{array}%
\right] _{p,q} x^k$.\newline

Now we have,
\begin{equation*}
[k]_{p,q}=p^{k-1}+q[k-1]_{p,q},~~\mbox{and}~~
[k]_{p,q}^2=q[k]_{p,q}[k-1]_{p,q}+p^{k-1}[k]_{p,q},
\end{equation*}
using it in above, we get
\begin{equation*}
L_n^{p,q}\left(\frac{t^2}{(1+t)^2};x\right) =\frac{pq}{\ell_n^{p,q}(x)}%
\sum_{k=2}^n \frac{q[k]_{p,q}[k-1]_{p,q}p^{2n-2k+2}}{[n+1]_{p,q}^2} p^{\frac{%
(n-k)(n-k-1)}{2}}q^{\frac{k(k-1)}{2}} \left[
\begin{array}{c}
n \\
k%
\end{array}%
\right] _{p,q} x^k
\end{equation*}
\begin{equation*}
+\frac{pq}{\ell_n^{p,q}(x)}\sum_{k=1}^n p^{k-1}\frac{[k]_{p,q}p^{2n-2k+2}}{%
[n+1]_{p,q}^2} p^{\frac{(n-k)(n-k-1)}{2}}q^{\frac{k(k-1)}{2}} \left[
\begin{array}{c}
n \\
k%
\end{array}%
\right] _{p,q} x^k
\end{equation*}
\begin{equation*}
=\frac{pq}{\ell_n^{p,q}(x)} \frac{q[n]_{p,q}[n-1]_{p,q}}{[n+1]_{p,q}^2}
\sum_{k=2}^n p^{\left((2n-2k+2)+\frac{(n-k)(n-k-1)}{2}\right)}q^{\frac{k(k-1)%
}{2}} \left[
\begin{array}{c}
n-2 \\
k-2%
\end{array}%
\right] _{p,q} x^k
\end{equation*}
\begin{equation*}
+\frac{pq}{\ell_n^{p,q}(x)} \frac{[n]_{p,q}}{[n+1]_{p,q}^2}\sum_{k=1}^n
p^{\left((k-1)+(2n-2k+2)+\frac{(n-k)(n-k-1)}{2}\right)}q^{\frac{k(k-1)}{2}} %
\left[
\begin{array}{c}
n-1 \\
k-1%
\end{array}%
\right] _{p,q} x^k
\end{equation*}
\begin{equation*}
=x^2\frac{pq}{\ell_n^{p,q}(x)} \frac{q[n]_{p,q}[n-1]_{p,q}}{[n+1]_{p,q}^2}
\sum_{k=0}^{n-2} p^{\left((2n-2k-2)+\frac{(n-k-2)(n-k-3)}{2}\right)}q^{\frac{%
(k+1)(k+2)}{2}} \left[
\begin{array}{c}
n-2 \\
k%
\end{array}%
\right] _{p,q} x^k
\end{equation*}
\begin{equation*}
+x\frac{pq}{\ell_n^{p,q}(x)} \frac{[n]_{p,q}}{[n+1]_{p,q}^2}\sum_{k=0}^{n-1}
p^{\left(k+(2n-2k)+\frac{(n-k-1)(n-k-2)}{2}\right)}q^{\frac{k(k+1)}{2}} %
\left[
\begin{array}{c}
n-1 \\
k%
\end{array}%
\right] _{p,q} x^k
\end{equation*}
\begin{equation*}
=x^2\frac{pq}{\ell_n^{p,q}(x)} \frac{pq^2[n]_{p,q}[n-1]_{p,q}}{[n+1]_{p,q}^2}
\sum_{k=0}^{n-2} p^{\frac{(n-k)(n-k-1)}{2}}q^{\frac{k(k-1)}{2}} \left[
\begin{array}{c}
n-2 \\
k%
\end{array}%
\right] _{p,q} (q^2x)^k
\end{equation*}
\begin{equation*}
+x\frac{pq}{\ell_n^{p,q}(x)} \frac{p^{n+1}[n]_{p,q}}{[n+1]_{p,q}^2}%
\sum_{k=0}^{n-1} p^{\frac{(n-k)(n-k-1)}{2}}q^{\frac{k(k-1)}{2}} \left[
\begin{array}{c}
n-1 \\
k%
\end{array}%
\right] _{p,q} (qx)^k
\end{equation*}
\begin{equation*}
=\frac{p^2q^3[n]_{p,q}[n-1]_{p,q}}{[n+1]_{p,q}^2}\frac{x^2}{(1+x)(p+qx)} +%
\frac{p^{n+2}q[n]_{p,q}}{[n+1]_{p,q}^2}\left(\frac{x}{1+x}\right).
\end{equation*}
This proves (3).
\end{enumerate}
\end{proof}
\textbf{Korovkin's type approximation properties}

In this section, we obtain the Korovkin's type statistical approximation properties for
the operators defined by \eqref{nas3}, using Theorem \ref{main}.\newline

\parindent=8mm In order to obtain the convergence results for the operators $%
L_{n}^{p,q}$, we take $q=q_{n},~~p=p_{n}$ where $q_{n}\in (0,1)$ and $%
p_{n}\in (q_{n},1]$ satisfy
\begin{align}\label{nas5}
\lim_{n}p_{n}=1,~~~~~~\lim_{n}q_{n}=1
\end{align}

\begin{theorem}
Let $p=p_n and q=q_n$ satisfy \eqref{nas5} for $0<q_n<p_n\leq 1$, and if $%
L_n^{p_n,q_n}$ is defined by \eqref{nas3}, then for any function $f \in H_\omega$,
\begin{equation*}
st-\lim_n \parallel L_n^{p_n,q_n}(f;x)-f \parallel_{C_{B}}=0.
\end{equation*}
\end{theorem}

\begin{proof}
In the light of Theorem \ref{main}, it is sufficient to prove the followings:
\begin{align}\label{nas6}
st-\lim_{n\rightarrow \infty }\parallel L_{n}^{p_{n},q_{n}}\left( \left( \frac{t%
}{1+t}\right) ^{\nu };x\right) -\left( \frac{x}{1+x}\right) ^{\nu }\parallel
_{C_{B}}=0,~~for ~~~\nu =0,1,2
\end{align}%
From Lemma \ref{main1}, the first condition of \eqref{nas6} is easily obtained for $\nu =0$. Also,
we can easily see from (\ref{main2}) of Lemma \ref{main1} that \newline
\begin{eqnarray*}
\parallel L_{n}^{p_{n},q_{n}}\left( \left( \frac{t}{1+t}\right) ^{\nu
};x\right) -\left( \frac{x}{1+x}\right) ^{\nu }\parallel _{C_{B}} &\leq &%
\bigg{|}\frac{p_{n}q_{n}[n]_{p_{n},q_{n}}}{[n+1]_{p_{n},q_{n}}}-1\bigg{|} \\
&=& 1-p_{n}q_{n}\frac{[n]_{{p_{n}},{q_{n}}}}{[n+1]_{p_{n},{q_{n}}}}.
\end{eqnarray*}%
Now for a given $\varepsilon>0$, we define the following sets
\begin{equation*}
U= \bigg\{n: \|L_n^{p_n, q_n}\left(\frac{t}{1+t};x\right)-\frac{x}{1+x}\|\geqslant \varepsilon\bigg\},
\end{equation*}

\begin{equation*}
U_1= \bigg\{n:1-p_{n}q_{n}\frac{[n]_{{p_{n}},{q_{n}}}}{[n+1]_{p_{n},{q_{n}}}}\geq\varepsilon\bigg\}.
\end{equation*}
It is obvious that $U \subset U_1$, so we have
\begin{equation*}
\delta\{ k\leq n:\|L_n^{p_n, q_n}\left(\frac{t}{1+t};x\right)-\frac{x}{1+x}\| \geq\varepsilon\}
\leqslant\delta \{k\leq n:1-p_{n}q_{n}\frac{[n]_{{p_{n}},{q_{n}}}}{[n+1]_{p_{n},{q_{n}}}}\geq\varepsilon\}.
\end{equation*}
Now using (2.1) it is clear that
\begin{equation*}
st-\lim\limits_{n}\left( 1-p_{n}q_{n}\frac{[n]_{{p_{n}},{q_{n}}}}{[n+1]_{p_{n},{q_{n}}}}\right) =0,
\end{equation*}
so
\begin{equation*}
\delta \{k\leq n:1-p_{n}q_{n}\frac{[n]_{{p_{n}},{q_{n}}}}{[n+1]_{p_{n},{q_{n}}}}\geq\varepsilon\}=0,
\end{equation*}
then
\begin{equation*}
st-\lim\limits_{n}\bigg\|L_n^{p_n, q_n}\left(\frac{t}{1+t};x\right)-\frac{x}{1+x}\bigg\|_{C_B} =0.
\end{equation*}

which proves that the condition \eqref{nas6} holds
for $\nu =1$.
To verify this condition for $\nu =2$, consider (\ref{main4}) of Lemma
\ref{main1}. Then, we see that\newline
\newline
$\bigg\| L_{n}^{p_{n},q_{n}}\left( \left( \frac{t}{1+t}\right)
^{2}; x\right) -\left( \frac{x}{1+x}\right) ^{2}\bigg\|
_{C_{B}}$ \\ \newline
$=\sup_{x\geq 0}\left\{ \frac{x^{2}}{(1+x)^{2}}\left( \frac{%
p_n^{2}q_n^3[n]_{p_{n},q_{n}}[n-1]_{p_{n},q_{n}}}{[n+1]_{p_{n},q_{n}}^{2}}.%
\frac{1+x}{p_{n}+q_{n}x}-1\right) +\frac{p_{n}^{n+2}q[n]_{p_{n},q_{n}}}{%
[n+1]_{p_{n},q_{n}}^{2}}\cdot \frac{x}{1+x}\right\} $.\newline
After some calculations, we get
\begin{equation*}
\frac{\lbrack n]_{p_{n},q_{n}}[n-1]_{p_{n},q_{n}}}{[n+1]_{p_{n},q_{n}}^{2}}=%
\frac{1}{q_n^3}\left\{ 1-p_{n}^{n}\left( 2+\frac{q_{n}}{p_{n}}\right)
\frac{1}{[n+1]_{p_{n},q_{n}}}+(p_{n}^{n})^{2}\left( 1+\frac{q_{n}}{p_{n}}%
\right) \frac{1}{[n+1]_{p_{n},q_{n}}^{2}}\right\} ,
\end{equation*}%
and
\begin{equation*}
\frac{\lbrack n]_{p_{n},q_{n}}}{[n+1]_{p_{n},q_{n}}^{2}}=\frac{1}{q_{n}}%
\left( \frac{1}{[n+1]_{p_{n},q_{n}}}-p_{n}^{n}\frac{1}{%
[n+1]_{p_{n},q_{n}}^{2}}\right) .
\end{equation*}%
Then we are led to \newline
$\bigg\| L_{n}^{p_{n},q_{n}}\left( \left( \frac{t}{1+t}\right)
^{2}; x\right) -\left( \frac{x}{1+x}\right) ^{2}\bigg\| _{C_{B}}$ \\ \newline
$\leq | (p_n^2 - 1)+ p_n^{n+2} \left( \frac{-1}{[n+1]_{p_n, q_n}} + \frac{p_n^n}{[n+1]_{p_n, q_n}^2}\right)+ p_n^{n+1} \left( \frac{-q_n}{[n+1]_{p_n, q_n}}+ \frac {p_n^2}{[n+1]_{p_n, q_n}} + \frac{p_n^n q_n}{[n+1]_{p_n, q_n}} \right)|$. \newline
$= 1-p_n^2 +p_n^{n+2} \left(\frac{1}{[n+1]_{p_n, q_n}} -\frac{p_n^n}{[n+1]_{p_n, q_n}^2}\right)+  p_n^{n+1} \left( \frac{q_n}{[n+1]_{p_n, q_n}}- \frac {p_n^2+ p_n^n q_n }{[n+1]_{p_n, q_n}^2}\right)$. \\

Now if we denote $1-p_n^2 $, $p_n^{n+2}\left( \frac{1}{[n+1]_{p_n, q_n}}- \frac{p_n^n}{[n+1]_{p_n, q_n}^2}\right)$ and $p_n^{n+1} \left( \frac{q_n}{[n+1]_{p_n, q_n}}- \frac{p_n^2 + p_n^n q_n}{ [n+1]_{p_n, q_n}^2}\right)$ by $\alpha_n$, $\beta_n$ and $\gamma_n$ respectively, then by using (2.1), we find that
\begin{equation}
st-\lim\limits_{n}\alpha_n = 0, st-\lim\limits_{n}\beta_n = 0 and st-\lim\limits_{n}\gamma_n = 0.
\end{equation}
Now for a given $\varepsilon>0$, we define the following sets
\begin{equation*}
U = \bigg\{n : \bigg\| L_{n}^{p_{n},q_{n}}\left( \left( \frac{t}{1+t}\right)
^{2}; x\right) -\left( \frac{x}{1+x}\right) ^{2}\bigg\|_{C_B}\geq \varepsilon \bigg\},
\end{equation*}
$U_1 = \bigg\{n : \alpha_n \geq \frac{\varepsilon}{3} \bigg\}$,
$U_2 = \bigg\{n : \beta_n \geq \frac{\varepsilon}{3} \bigg\}$ and
$U_3 = \bigg\{n : \gamma_n \geq \frac{\varepsilon}{3} \bigg\}$.
It is obvious that $U\subseteq U_1\cup U_2 \cup U_3$. So we have
\begin{eqnarray*}
\delta\ \bigg\{ k\leq n:\bigg\|L_n^{p_n, q_n}\left(\left(\frac{t}{1+t}\right)^2;x\right)-\left(\frac{x}{1+x}\right)^2\bigg\|_{C_B} \geq\varepsilon \bigg\}\\
\leqslant\delta\bigg\{k\leq n:\alpha_n \geq \frac{\varepsilon}{3}\bigg\} +\delta\bigg\{k\leq n:\beta_n \geq \frac{\varepsilon}{3}\bigg\} \newline
~~~~+ \delta\bigg\{k\leq n:\gamma_n \geq \frac{\varepsilon}{3}\bigg\}.
\end{eqnarray*}
Now by virtue of (2.3), the right hand side of the above inequality is trivial, so we get
\begin{equation*}
st-\lim\limits_{n}\bigg\|L_n^{p_n, q_n}\left(\left(\frac{t}{1+t}\right)^2;x\right)-\left(\frac{x}{1+x}\right)^2\bigg\|_{C_B}= 0.
\end{equation*}
Hence the proof of the theorem is complete.
\end{proof}

\section{Rate of Convergence}

In this section, we calculate the rate of convergence of the operators \eqref{nas3} by
means of modulus of continuity and Lipschitz type maximal functions.

\parindent=8mmThe modulus of continuity for $f\in H_{\omega }$ is defined by
\begin{equation*}
\widetilde{\omega }(f;\delta )=\sum_{\substack{ \mid \frac{t}{1+t}-\frac{x}{%
1+x}\mid \leq \delta ,  \\ x,t\geq 0}}\mid f(t)-f(x)\mid
\end{equation*}%
where $\widetilde{\omega }(f;\delta )$ satisfies the following conditions.
For all $f\in H_{\omega }(\mathbb{R}_{+})$

\begin{enumerate}
\item $\lim_{\delta \to 0}\widetilde{\omega}(f; \delta)=0$

\item $\mid f(t)-f(x) \mid \leq \widetilde{\omega}(f; \delta)
\left( \frac{\mid \frac{t}{1+t}-\frac{x}{1+x}\mid}{\delta}+1 \right)$
\end{enumerate}

\begin{theorem}\label{main9}
Let $p=p_{n} and q=q_{n}$ satisfy \eqref{nas5}, for $0<q_{n}<p_{n}\leq 1$, and let $%
L_{n}^{p_{n},q_{n}}$ be defined by \eqref{nas3}. Then for each $x\geq 0$ and for
any function $f\in H_{\omega }$, we have
\begin{equation*}
\mid L_{n}^{p_{n},q_{n}}(f;x)-f\mid \leq 2\widetilde{\omega }(f;\sqrt{\delta
_{n}(x)}),
\end{equation*}%
where
\begin{equation*}
\delta _{n}(x)=\frac{x^{2}}{(1+x)^{2}}\left( \frac{%
p_{n}^{2}q_{n}^{3}[n]_{p_{n},q_{n}}[n-1]_{p_{n},q_{n}}}{[n+1]_{p_{n},q_{n}}^{2}}%
\frac{1+x}{p_{n}+q_{n}x}-2\frac{p_{n} q_{n}[n]_{p_{n},q_{n}}}{[n+1]_{p_{n},q_{n}}}%
+1\right) +\frac{p_{n}^{n+2} q_n[n]_{p_{n},q_{n}}}{[n+1]_{p_{n},q_{n}}^{2}}\frac{%
x}{1+x}.
\end{equation*}
\end{theorem}

\begin{proof}
\begin{eqnarray*}
\mid L_n^{p_n,q_n}(f;x)-f \mid &\leq & L_n^{p_n,q_n}\left(\mid f(t)-f(x)
\mid;x \right) \\
&\leq & \widetilde{\omega}(f; \delta) \left\{1+\frac{1}{\delta}
L_n^{p_n,q_n} \left( \bigg{|} \frac{t}{1+t}-\frac{x}{1+x}\big{|};x
\right)\right\}.
\end{eqnarray*}
Now by using the Cauchy-Schwarz inequality, we have
\begin{eqnarray*}
\mid L_n^{p_n,q_n}(f;x)-f \mid &\leq & \widetilde{\omega}(f; \delta_n)
\left\{1+\frac{1}{\delta_n} \left[ \left( L_n^{p_n,q_n} \left(\frac{t}{1+t}-%
\frac{x}{1+x}\right)^2;x \right) \right]^{\frac{1}{2}} \left(
L_n^{p_n,q_n}(1;x)\right)^{\frac{1}{2}}\right\}
\end{eqnarray*}
$\leq \widetilde{\omega}(f; \delta_n) \left\{ 1+ \frac{1}{\delta_n} \left[
\frac{x^2}{(1+x)^2}\left( \frac{p_n^2q_n^3[n]_{p_n,q_n}[n-1]_{p_n,q_n}}{%
[n+1]_{p_n,q_n}^2} \frac{1+x}{p_n+q_nx} -2\frac{p_nq_n[n]_{p_n,q_n}}{%
[n+1]_{p_n,q_n}}+1\right)+\frac{p_n^{n+2}q_n[n]_{p_n,q_n}}{[n+1]_{p_n,q_n}^2}%
\frac{x}{1+x}\right]^{\frac{1}{2}}\right\}$.\newline

\parindent=8mmThis completes the proof.
\end{proof}
\parindent=8mmNow we will give an estimate concerning the rate of
convergence by means of Lipschitz type maximal functions. In \cite{aral1},
the Lipschitz type maximal function space on $E\subset \mathbb{R}_{+}$ is
defined as
\begin{align}\label{nas7}
\widetilde{W}_{\alpha ,E}=\{f:\sup (1+x)^{\alpha }\widetilde{f}_{\alpha
}(x)\leq M\frac{1}{(1+y)^{\alpha }}:x\leq 0,~\mbox{and}~y\in
E\}
\end{align}%
where $f$ is bounded and continuous function on $\mathbb{R}_{+}$, $M$ is a
positive constant and $0<\alpha \leq 1$.

In \cite{lenz}, B. Lenze introduced a Lipschitz type maximal function $%
f_{\alpha }$ as follows:
\begin{align}\label{nas8}
f_{\alpha }(x,t)=\sum_{\substack{ t>0  \\ t\neq x}}\frac{\mid f(t)-f(x)\mid
}{\mid x-t\mid ^{\alpha }}.
\end{align}%
We denote by $d(x,E)$, the distance between $x$ and $E$, that is
\begin{equation*}
d(x,E)=\inf \{\mid x-y\mid ;y\in E\}.
\end{equation*}

\begin{theorem}\label{main8}
For all $f \in \widetilde{W}_{\alpha, E},$ we have
\begin{align}\label{nas9}
\mid L_n^{p_n,q_n}(f;x) -f(x) \mid \leq M\left( \delta_n^{\frac{\alpha}{2}%
}(x)+2 \left( d(x,E)\right)^\alpha\right)
\end{align}
where $\delta_n(x)$ is defined as in Theorem \ref{main9}.
\end{theorem}

\begin{proof}
Let $\overline{E}$ denote the closure of the set $E$. Then there exits a $%
x_{0}\in \overline{E}$ such that $\mid x-x_{0}\mid =d(x,E)$, where $x\in
\mathbb{R}_{+}$. Thus we can write
\begin{equation*}
\mid f-f(x)\mid \leq \mid f-f(x_{0})\mid +\mid f(x_{0})-f(x)\mid .
\end{equation*}%
Since $L_{n}^{p_{n},q_{n}}$ are positive linear operators, so for $f\in \widetilde{W%
}_{\alpha ,E}$, by using the previous inequality, we have
\begin{equation*}
\mid L_{n}^{p_{n},q_{n}}(f;x)-f(x)\mid \leq \mid L_{n}^{p_{n},q_{n}}(\mid
f-f(x_{0})\mid ;x)+\mid f(x_{0})-f(x)\mid L_{n}^{p_{n},q_{n}}(1;x)
\end{equation*}%
\begin{equation*}
\leq M\left( L_{n}^{p_{n},q_{n}}\left( \bigg{|}\frac{t}{1+t}-\frac{x_{0}}{%
1+x_{0}}\bigg{|}^{\alpha };x\right) +\frac{\mid x-x_{0}\mid ^{\alpha }}{%
(1+x)^{\alpha }(1+x_{0})^{\alpha }}L_{n}^{p_{n},q_{n}}(1;x)\right) .
\end{equation*}%
Now $(a+b)^{\alpha }\leq a^{\alpha }+b^{\alpha }$ consequently
implies that
\begin{equation*}
L_{n}^{p_{n},q_{n}}\left( \bigg{|}\frac{t}{1+t}-\frac{x_{0}}{1+x_{0}}\bigg{|}%
^{\alpha };x\right) \leq L_{n}^{p_{n},q_{n}}\left( \bigg{|}\frac{t}{1+t}-%
\frac{x}{1+x}\bigg{|}^{\alpha };x\right) +L_{n}^{p_{n},q_{n}}\left( \bigg{|}%
\frac{x}{1+x}-\frac{x_{0}}{1+x_{0}}\bigg{|}^{\alpha };x\right)
\end{equation*}

\begin{equation*}
L_{n}^{p_{n},q_{n}}\left( \bigg{|}\frac{t}{1+t}-\frac{x_{0}}{1+x_{0}}\bigg{|}%
^{\alpha };x\right) \leq L_{n}^{p_{n},q_{n}}\left( \bigg{|}\frac{t}{1+t}-%
\frac{x}{1+x}\bigg{|}^{\alpha };x\right) +\frac{\mid x-x_{0}\mid ^{\alpha }}{%
(1+x)^{\alpha }(1+x_{0})^{\alpha }}L_{n}^{p_{n},q_{n}}(1;x).
\end{equation*}%
By using the H\"{o}lder's inequality with $p=\frac{2}{\alpha }$ and $q=\frac{2}{%
2-\alpha }$, we have\\ \newline
$L_{n}^{p_{n},q_{n}}\left( \bigg{|}\frac{t}{1+t}-\frac{x_{0}}{1+x_{0}}%
\bigg{|}^{\alpha };x\right) \leq L_{n}^{p_{n},q_{n}}\left( \left( \frac{t}{%
1+t}-\frac{x}{1+x}\right) ^{2};x\right) ^{\frac{\alpha }{2}%
}(L_{n}^{p_{n},q_{n}}(1;x))^{\frac{2-\alpha }{2}}$
\begin{equation*}
+\frac{\mid x-x_{0}\mid
^{\alpha }}{(1+x)^{\alpha }(1+x_{0})^{\alpha }}L_{n}^{p_{n},q_{n}}(1;x)
\end{equation*}

\begin{equation*}
=\delta _{n}^{\frac{\alpha }{2}}(x)+\frac{\mid x-x_{0}\mid ^{\alpha }}{%
(1+x)^{\alpha }(1+x_{0})^{\alpha }}.
\end{equation*}

\parindent=8mmThis completes the proof.
\end{proof}
\begin{corollary}\label{main12}
If we take $E= \mathbb{R}_+$ as a particular case of Theorem \ref{main8}, then for
all $f \in \widetilde{W}_{\alpha, \mathbb{R}_+}$, we have
\begin{equation*}
\mid L_n^{p_n,q_n}(f;x) -f(x) \mid \leq M \delta_n^{\frac{\alpha}{2}}(x),
\end{equation*}
\end{corollary}

where $\delta_n(x)$ is defined as in Theorem \ref{main9}.
\begin{theorem}
If $x\in (0,\infty )\backslash \left\{ p^{n-k+1}\frac{[k]_{p,q}}{%
[n-k+1]_{p,q}q^{k}}\bigg{|}k=0,1,2,\cdots ,n\right\} $, then

$L_n^{p,q}(f;x)-f\left(\frac{px}{q}\right)=-\frac{x^{n+1}}{\ell_n^{p,q}(x)} %
\left[\frac{px}{q};\frac{p[n]_{p,q}}{q^n};f\right] pq^{\frac{n(n-1)}{2}-n}$
\begin{align}\label{nas10}
+\frac{x}{\ell_n^{p,q}(x)}\sum_{k=0}^{n-1} \left[\frac{px}{q};p^{n-k+1}\frac{%
[k]_{p,q}}{[n-k+1]_{p,q}q^{k} };f\right] \frac{1}{[n-k]_{p,q}} p^{\frac{%
(n-k)(n-k-1)}{2}-(k-n)-1}q^{\frac{k(k-1)}{2}-k} \left[
\begin{array}{c}
n \\
k%
\end{array}%
\right] _{p,q} x^{k}.
\end{align}
\end{theorem}

\begin{proof}
By using \eqref{nas3}, we have\\ \newline
$L_n^{p,q}(f;x)-f\left(\frac{px}{q}\right)=\frac{pq}{\ell_n^{p,q}(x)}%
\sum_{k=0}^n \bigg\{f \left( \frac{ p^{n-k+1}[k]_{p,q}}{[n-k+1]_{p,q}q^k }%
\right) -f\left(\frac{px}{q}\right)\bigg\} p^{\frac{(n-k)(n-k-1)}{2}}q^{%
\frac{k(k-1)}{2}}
\left[
\begin{array}{c}
n \\
k%
\end{array}%
\right] _{p,q} x^k $\\
$=-\frac{1}{\ell_n^{p,q}(x)}\sum_{k=0}^n \left(\frac{px}{q}- \frac{
p^{n-k+1}[k]_{p,q}}{[n-k+1]_{p,q}q^k }\right) \left[\frac{px}{q}; \frac{%
p^{n-k+1}[k]_{p,q}}{[n-k+1]_{p,q}q^k };f\right] p^{\frac{(n-k)(n-k-1)}{2}+1}q^{%
\frac{k(k-1)}{2}+1} \left[
\begin{array}{c}
n \\
k%
\end{array}%
\right] _{p,q} x^k$\\ \newline
By using $\frac{[k]_{p,q}}{[n-k+1]_{p,q}}\left[
\begin{array}{c}
n \\
k%
\end{array}%
\right] _{p,q}=\left[
\begin{array}{c}
n \\
k-1%
\end{array}%
\right] _{p,q}$, we have

\begin{equation*}
L_n^{p,q}(f;x)-f\left(\frac{px}{q}\right)=-\frac{x}{\ell_n^{p,q}(x)}%
\sum_{k=0}^n \left[\frac{px}{q};\frac{ p^{n-k+1}[k]_{p,q}}{[n-k+1]_{p,q}q^k }%
;f\right] p^{\frac{(n-k)(n-k-1)}{2}+2}q^{\frac{k(k-1)}{2}} \left[
\begin{array}{c}
n \\
k%
\end{array}%
\right] _{p,q} x^k.
\end{equation*}

\begin{equation*}
+\frac{1}{\ell_n^{p,q}(x)}\sum_{k=1}^n \left[\frac{px}{q};\frac{ p^{n-k+1}
[k]_{p,q}}{[n-k+1]_{p,q}q^k };f\right] p^{\frac{(n-k)(n-k-1)}{2}-(k-n-1)-1}q^{%
\frac{k(k-1)}{2}-(k-1)} \left[
\begin{array}{c}
n \\
k-1%
\end{array}%
\right] _{p,q} x^k
\end{equation*}

\begin{equation*}
=-\frac{x}{\ell_n^{p,q}(x)}\sum_{k=0}^n \left[\frac{px}{q};\frac{ p^{n-k+1}
[k]_{p,q}}{[n-k+1]_{p,q}q^k };f\right]  p^{\frac{(n-k)(n-k-1)}{2}+2}q^{\frac{k(k-1)}{2}} \left[
\begin{array}{c}
n \\
k%
\end{array}%
\right] _{p,q} x^k.
\end{equation*}

\begin{equation*}
+\frac{x}{\ell_n^{p,q}(x)}\sum_{k=0}^{n-1} \left[\frac{px}{q};\frac{
p^{n-k}[k+1]_{p,q}}{[n-k]_{p,q}q^{k+1} };f\right] p^{\frac{(n-k)(n-k-1)}{2}-(k-n-1)-2}%
q^{\frac{k(k-1)}{2}-k} \left[
\begin{array}{c}
n \\
k%
\end{array}%
\right] _{p,q} x^{k}
\end{equation*}

$=-\frac{x^{n+1}}{\ell_n^{p,q}(x)} \left[\frac{px}{q};\frac{p[n]_{p,q}}{q^n}%
;f\right] pq^{\frac{n(n-1)}{2}-n} $\\ \newline
$+\frac{x}{\ell_n^{p,q}(x)}\sum_{k=0}^{n-1}\left \{ \left[\frac{px}{q};\frac{
p^{n-k}[k+1]_{p,q}}{[n-k]_{p,q}q^{k+1} };f\right] -\left[\frac{px}{q};\frac{
p^{n-k+1} [k]_{p,q}}{[n-k+1]_{p,q}q^{k} };f\right] \right \} p^{\frac{%
(n-k)(n-k-1)}{2}-(k-n-1)-2}q^{\frac{k(k-1)}{2}-k} \left[
\begin{array}{c}
n \\
k%
\end{array}%
\right] _{p,q} x^{k}.$\\ \newline

Now by using the results\\ \newline
$\left[\frac{px}{q};\frac{ p^{n-k}[k+1]_{p,q}}{[n-k]_{p,q}q^{k+1} };f\right]
-\left[\frac{px}{q};\frac{ p^{n-k+1} [k]_{p,q}}{[n-k+1]_{p,q}q^{k} };f\right]
$
\begin{equation*}
= \left( \frac{ p^{n-k}[k+1]_{p,q}}{[n-k]_{p,q}q^{k+1} }-\frac{ p^{n-k+1}
[k]_{p,q}}{[n-k+1]_{p,q}q^{k} }\right) \bigg\{ \frac{px}{q};\frac{
p^{n-k+1} [k]_{p,q}}{[n-k+1]_{p,q}q^{k} };\frac{ p^{n-k} [k+1]_{p,q}}{%
[n-k]_{p,q}q^{k+1} };f \bigg\}
\end{equation*}
and
\begin{equation*}
\frac{ p^{n-k} [k+1]_{p,q}}{[n-k]_{p,q}q^{k+1} }-\frac{ p^{n-k+1} [k]_{p,q}}{%
[n-k+1]_{p,q}q^{k} }=[n+1]_{p,q},
\end{equation*}
we have\newline
\newline
$L_n^{p,q}(f;x)-f\left(\frac{px}{q}\right)=-\frac{x^{n+1}}{\ell_n^{p,q}(x)} %
\left[\frac{px}{q};\frac{p[n]_{p,q}}{q^n};f\right] pq^{\frac{n(n-1)}{2}-n} $%
\newline
$+\frac{x}{\ell_n^{p,q}(x)}\sum_{k=0}^{n-1}\left \{ \left[\frac{px}{q};\frac{
p^{n-k+1} [k]_{p,q}}{[n-k+1]_{p,q}q^{k} };f\right] \frac{ p^{n-k}[n+1]_{p,q}%
}{[n-k]_{p,q}[n-k+1]_{p,q}q^{k+1} } \right\} p^{\frac{(n-k)(n-k-1)}{2}-(k-n)-1}q^{%
\frac{k(k-1)}{2}-k} \left[
\begin{array}{c}
n \\
k%
\end{array}%
\right] _{p,q} x^{k}.$\newline

\parindent=8mm which completes the proof.
\end{proof}
\section{Some Generalizations of $L_n^{p,q}$}

In this section, we present some generalizations of the operators $L_n^{p,q}$
based on $(p,q)$-integers similar to work done in \cite{n1, aral1}.

We consider a sequence of linear positive operators based on $(p,q)$%
-integers as follows:
\begin{align}\label{nas11}
L_{n }^{(p,q),\gamma}(f;x)=\frac{pq}{\ell _{n}^{p,q}(x)}\sum_{k=0}^{n}f\left(
\frac{p^{n-k+1}[k]_{p,q}+\gamma }{b_{n,k}}\right) p^{\frac{(n-k)(n-k-1)}{2}%
}q^{\frac{k(k-1)}{2}}\left[
\begin{array}{c}
n \\
k%
\end{array}%
\right] _{p,q}x^{k},~~~~~~(\gamma \in \mathbb{R})
\end{align}%
where $b_{n,k}$ satisfy the following conditions:
\begin{equation*}
p^{n-k+1}[k]_{p,q}+b_{n,k}=c_{n}~~~~~~\mbox{and}~~~~\frac{[n]_{p,q}}{c_{n}}%
\rightarrow 1~~~~~~\mbox{for}~~~~n\rightarrow \infty .
\end{equation*}%
It is easy to check that if $b_{n,k}=q^{k}[n-k+1]_{p,q}+\beta $ for any $n,k$
and $0<q<p\leq 1$, then $c_{n}=[n+1]_{p,q}+\beta $. If we choose $p=1$, then the 
operators reduce to the generalization of $q$-BBH opeartors defined in \cite%
{aral1}, and which turn out to be D. D. Stancu-type generalization of
Bleimann, Butzer, and Hahn operators based on $q$-integers \cite{n2}. If we
choose $\gamma =0,~~~q=1$ as in \cite{aral1} for $p=1$, then the operators
become the special case of the Balázs-type generalization of the $q$-BBH
operators \cite{aral1} given in \cite{n1}.

\begin{theorem}
Let $p=p_n~and~q=q_n$ satisfy \eqref{nas5} for $0<q_n<p_n\leq 1$ and let $%
L_n^{(p_n,q_n), \gamma}$ be defined by \eqref{nas11}. Then for any function $f \in
\widetilde{W}_{\alpha, [0,\infty)}$, we have\\ \newline
$\lim_{n}\parallel L_{n }^{(p_{n},q_{n}),\gamma}(f;x)-f(x)\parallel _{C_{B}}\leq 3M $\\%
\newline
$\times \max \left\{ \left( \frac{[n]_{p_{n},q_{n}}}{c_{n}+\gamma }\right)
^{\alpha }\left( \frac{\gamma }{[n]_{p_{n},q_{n}}}\right) ^{\alpha },%
\bigg{|}1-\frac{[n+1]_{p_{n},q_{n}}}{c_{n}+\gamma }\bigg{|}^{\alpha }\left(
\frac{p_{n}q_{n}[n]_{p_{n},q_{n}}}{[n+1]_{p_{n},q_{n}}}\right) ^{\alpha },1-2%
\frac{p_{n}q_{n}[n]_{p_{n},q_{n}}}{[n+1]_{p_{n},q_{n}}}+\frac{%
p_{n}q_{n}[n]_{p_{n},q_{n}}[n-1]_{p_{n},q_{n}}}{[n+1]_{p_{n},q_{n}}^{2}}\right\}.$
\end{theorem}
\begin{proof}

Using \eqref{nas3} and \eqref{nas11}, we have \\ \newline
$\mid L_{n }^{(p,q),\gamma}(f;x)-f(x)\mid $ \\ \newline
$\leq \frac{pq}{\ell _{n}^{p_{n},q_{n}}(x)}\sum_{k=0}^{n}\bigg{|}f\left(
\frac{p_{n}^{n-k+1}[k]_{p_{n},q_{n}}+\gamma }{b_{n,k}}\right) -f\left( \frac{%
p_{n}^{n-k+1}[k]_{p_{n},q_{n}}}{\gamma +b_{n,k}}\right) \bigg{|}p_{n}^{\frac{%
(n-k)(n-k-1)}{2}}q_{n}^{\frac{k(k-1)}{2}}\left[
\begin{array}{c}
n \\
k%
\end{array}%
\right] _{p_{n},q_{n}}x^{k}$
\begin{equation*}
+ \frac{pq}{\ell_n^{p_n,q_n}(x)}\sum_{k=0}^n \bigg{|} f \left(\frac{%
p_n^{n-k+1}[k]_{p_n,q_n}}{\gamma+ b_{n,k}}\right) -f\left(\frac{%
p_n^{n-k+1}[k]_{p_n,q_n}}{[n-k+1]_{p_n,q_n}q_n^{k}}\right)\bigg{|} p_n^{%
\frac{(n-k)(n-k-1)}{2}}q_n^{\frac{k(k-1)}{2}} \left[
\begin{array}{c}
n \\
k%
\end{array}%
\right] _{p_n,q_n} x^k,
\end{equation*}
we have \\ \newline
$\mid L_{n }^{(p,q),\gamma}(f;x)-f(x)\mid $ \\ \newline
$\leq \frac{pq}{\ell _{n}^{p_{n},q_{n}}(x)}\sum_{k=0}^{n}\bigg{|}f\left(
\frac{p_{n}^{n-k+1}[k]_{p_{n},q_{n}}+\gamma }{b_{n,k}}\right) -f\left( \frac{%
p_{n}^{n-k+1}[k]_{p_{n},q_{n}}}{\gamma +b_{n,k}}\right) \bigg{|}p_{n}^{\frac{%
(n-k)(n-k-1)}{2}+1}q_{n}^{\frac{k(k-1)}{2}+1}\left[
\begin{array}{c}
n \\
k%
\end{array}%
\right] _{p_{n},q_{n}}x^{k}$
\begin{equation*}
+ \frac{pq}{\ell_n^{p_n,q_n}(x)}\sum_{k=0}^n \bigg{|} f \left(\frac{%
p_n^{n-k+1}[k]_{p_n,q_n}}{\gamma+ b_{n,k}}\right) -f\left(\frac{%
p_n^{n-k+1}[k]_{p_n,q_n}}{[n-k+1]_{p_n,q_n}q_n^{k}}\right)\bigg{|} p_n^{%
\frac{(n-k)(n-k-1)}{2}+1}q_n^{\frac{k(k-1)}{2}+1} \left[
\begin{array}{c}
n \\
k%
\end{array}%
\right] _{p_n,q_n} x^k
\end{equation*}
$+ \mid L_n^{p_n,q_n}(f;x) -f(x) \mid.$\\ \newline
Now for
$f\in \widetilde{W}_{\alpha ,[0,\infty )}$, by using the Corollary
\ref{main12}, 
we can write\\ \newline
$\mid L_{n }^{(p,q),\gamma}(f;x)-f(x)\mid $\\ \newline
$\leq \frac{M}{\ell _{n}^{p_{n},q_{n}}(x)}\sum_{k=0}^{n}\bigg{|}\frac{%
p_{n}^{n-k+1}[k]_{p_{n},q_{n}}+\gamma }{p_{n}^{n-k+1}[k]_{p_{n},q_{n}}+%
\gamma +b_{n,k}}-\frac{p_{n}^{n-k+1}[k]_{p_{n},q_{n}}}{\gamma
+p_{n}^{n-k+1}[k]_{p_{n},q_{n}}+b_{n,k}}\bigg{|}^{\alpha }p_{n}^{\frac{%
(n-k)(n-k-1)}{2}+1}q_{n}^{\frac{k(k-1)}{2}+1}\left[
\begin{array}{c}
n \\
k%
\end{array}%
\right] _{p_{n},q_{n}}x^{k}$\\
$+\frac{M}{\ell _{n}^{p_{n},q_{n}}(x)}\sum_{k=0}^{n}\bigg{|}\frac{%
p_{n}^{n-k+1}[k]_{p_{n},q_{n}}}{p_{n}^{n-k+1}[k]_{p_{n},q_{n}}+\gamma
+b_{n,k}}-\frac{p_{n}^{n-k+1}[k]_{p_{n},q_{n}}}{%
p_{n}^{n-k+1}[k]_{p_{n},q_{n}}+[n-k+1]_{p_{n},q_{n}}q_{n}^{k}}\bigg{|}$\\
$\times p_{n}^{%
\frac{(n-k)(n-k-1)}{2}+1}q_{n}^{\frac{k(k-1)}{2}+1}\left[
\begin{array}{c}
n \\
k%
\end{array}%
\right] _{p_{n},q_{n}}x^{k} +M\delta _{n}^{\frac{\alpha }{2}}(x).$\newline

This implies that \\ \newline
$\mid L_{n }^{(p,q),\gamma}(f;x)-f(x)\mid \leq M\left( \frac{[n]_{p_{n},q_{n}}%
}{c_{n}+\gamma }\right) ^{\alpha }\left( \frac{\gamma }{[n]_{p_{n},q_{n}}}%
\right) ^{\alpha }$
\begin{equation*}
+\frac{M}{\ell _{n}^{p_{n},q_{n}}(x)}\bigg{|}1-\frac{[n+1]_{p_{n},q_{n}}}{%
c_{n}+\gamma }\bigg{|}^{\alpha }\sum_{k=0}^{n}\left( \frac{%
p_{n}^{n-k+1}[k]_{p_{n},q_{n}}}{[n+1]_{p_{n},q_{n}}}\right) ^{\alpha }p_{n}^{%
\frac{(n-k)(n-k-1)}{2}+1}q_{n}^{\frac{k(k-1)}{2}+1}\left[
\begin{array}{c}
n \\
k%
\end{array}%
\right] _{p_{n},q_{n}}x^{k}+M\delta _{n}^{\frac{\alpha }{2}}(x)
\end{equation*}

\begin{equation*}
=M\left( \frac{[n]_{p_{n},q_{n}}}{c_{n}+\gamma }\right) ^{\alpha }\left(
\frac{\gamma }{[n]_{p_{n},q_{n}}}\right) ^{\alpha }+M\bigg{|}1-\frac{%
[n+1]_{p_{n},q_{n}}}{c_{n}+\gamma }\bigg{|}^{\alpha
}L_{n}^{p_{n},q_{n}}\left( \left( \frac{t}{1+t}\right) ^{\alpha };x\right)
+M\delta _{n}^{\frac{\alpha }{2}}(x).
\end{equation*}%
Using the H\"{o}lder's inequality for $p=\frac{1}{\alpha },~~~q=\frac{1}{1-\alpha }
$, we get\\ \newline
$\mid L_{n }^{(p,q),\gamma}(f;x)-f(x)\mid $\\ \newline
$\leq M\left( \frac{[n]_{p_{n},q_{n}}}{c_{n}+\gamma }\right) ^{\alpha }\left(
\frac{\gamma }{[n]_{p_{n},q_{n}}}\right) ^{\alpha }+M\bigg{|}1-\frac{%
[n+1]_{p_{n},q_{n}}}{c_{n}+\gamma }\bigg{|}^{\alpha
}L_{n}^{p_{n},q_{n}}\left( \frac{t}{1+t};x\right) ^{\alpha }\left(
L_{n}^{p_{n},q_{n}}(1;x)\right) ^{1-\alpha }+M\delta _{n}^{\frac{\alpha }{2}%
}(x)$\\
$\leq M\left( \frac{[n]_{p_{n},q_{n}}}{c_{n}+\gamma }\right) ^{\alpha }\left(
\frac{\gamma }{[n]_{p_{n},q_{n}}}\right) ^{\alpha }+M\bigg{|}1-\frac{%
[n+1]_{p_{n},q_{n}}}{c_{n}+\gamma }\bigg{|}^{\alpha }\left( \frac{%
p_{n}q_{n}[n]_{p_{n},q_{n}}}{[n+1]_{p_{n},q_{n}}}\frac{x}{1+x}\right) ^{\alpha
}+M\delta _{n}^{\frac{\alpha }{2}}(x),$\\
which completes the proof.
\end{proof}
\section{Construction of the bivariate operators}

In what follows we construct the bivariate extension of the operators (1.3). We will introduce the statistical convergence of the operators to a function $f$ and investigate the statistical rate of convergence of these operators.

Let $\mathbb{R}_+^2 = [0, \infty)\times [0, \infty),  f: \mathbb{R}_+^2\rightarrow \mathbb{R}$ and $0< p_{n_1}, q_{n_1}; p_{n_2}, q_{n_2}\leq1$. Then we define the bivariate companion of the operators (1.3) as follows:
\begin{eqnarray*}
L_{n_1, n_2}(f; p_{n_1}, p_{n_2}; q_{n_1}, q_{n_2}; x, y)=
 \frac{p_{n_1}p_{n_2}q_{n_1}q_{n_2}}{l_{n_1}^{p_{n_1}, q_{n_1}}\times l_{n_2}^{p_{n_2}, q_{n_2}}}
 \sum_{k_1=0}^{n_1}
 \sum_{k_2=0}^{n_2}
  f( \frac{
p_{n_1}^{n_1-k_1+1}[k_1]_{p_{n_1},q_{n_1}}}{[n_1-k_1+1]_{p_{n_1},q_{n_1}}q_{n_1}^{k_1} },\\
 \frac{p_{n_2}^{n_2-k_2+1}[k_2]_{p_{n_2},q_{n_2}}}{[n_2-k_2+1]_{p_{n_2},q_{n_2}}q_{n_2}^{k_2}}) 
~p_{n_1}^{\frac{(n_1-k_1)(n_1-k_1-1)}{2}}p_{n_2}^{\frac{(n_2-k_2)(n_2-k_2-1)}{2}}
q_{n_1}^{\frac{k_1(k_1-1)}{2}}q_{n_2}^{\frac{k_2(k_2-1)}{2}}
\left[
\begin{array}{c}
n_1 \\
k_1%
\end{array}%
\right] _{p_{n_1},q_{n_1}}\\
\times\left[
\begin{array}{c}
n_2 \\
k_2%
\end{array}%
\right] _{p_{n_2},q_{n_2}} x^{k_1}y^{k_2},
\end{eqnarray*}
where
$l_{n_1}^{p_{n_1}, q_{n_1}} = \prod_{s=0}^{n_1-1} (p_{n_1}^s+ q_{n_1}^sx )$ and $l_{n_2}^{p_{n_2}, q_{n_2}} = \prod_{s=0}^{n_2-1} (p_{n_2}^s+ q_{n_2}^sy )$.

For $K = [0, \infty)\times [0, \infty)$, the modulus of continuity for the bivariate case is defined as
\begin{equation*}
\omega_2(g; \delta_1, \delta_2)= sup\{|g(u_1, v_1)-g(u_2, v_2)|: (u_1, v_1),(u_2, v_2)\in K ~and~ |u_1-u_2|\leq \delta_1, |v_1-v_2|\leq \delta_2 \},
\end{equation*}
where, for each $g\in H_{\omega_2}, ~\omega_2(g; \delta_1, \delta_2)$ satisfies
\begin{equation*}
|g(u_1, v_1)-g(u_2, v_2)| \leq \omega_2\left(g \bigg|\frac{u_1}{1+u_1}- \frac{u_2}{1+u_2}\bigg|, \bigg|\frac{v_1}{1+v_1}- \frac{v_2}{1+v_2}\bigg| \right).
\end{equation*}
For detailed study of modulus of continuity for the bivariate analogue one is referred to \cite{gal}.

The first Korovkin type theorem for the statistical approximation for the bivariate analogue of linear positive operators defined in the space $H_{\omega_2}$ was obtained by Erkus and Duman \cite{erkus} which is as follows.

\begin{theorem}
Let $\{L_n\}$ be a sequence of positive linear operators from $H_{\omega_2}$ into $C_B(K)$. Then, for each $g\in H_{\omega_2}$,
\begin{equation*}
st-\lim_n \parallel L_n(g)-g \parallel=0,
\end{equation*}
holds if the following is satisfied
\begin{equation*}
st-\lim_n \parallel L_n(g_j)-g_j \parallel=0,~~~~for~~ j=0, 1, 2, 3
\end{equation*}
where
\begin{equation}
g_0(u, v)=0, ~~g_1(u, v)= \frac{u}{1+u}, ~~g_2(u, v)= \frac{v}{1+v}, ~~g_3(u, v)= (\frac{u}{1+u})^2+ (\frac{v}{1+v})^2.
\end{equation}
\end{theorem}

To study the statistical convergence of the bivariate operators, the following lemma is essential.
\begin{lemma} The bivariate operators defined above satisfy the followings:
\begin{enumerate}
\item $L_{n_1, n_2}(f_0; p_{n_1}, p_{n_2}; q_{n_1}, q_{n_2}; x, y)= p_{n_1} p_{n_2}q_{n_1} q_{n_2}$,
\item $L_{n_1, n_2}(f_1; p_{n_1}, p_{n_2}; q_{n_1}, q_{n_2}; x, y)= p_{n_1} p_{n_2}q_{n_1} q_{n_2} \frac{[n_1]_{p_{n_1},q_{n_1}}}{[n_1+1]_{p_{n_1}, q_{n_1}}} \frac{x}{1+x} $,
\item $L_{n_1, n_2}(f_2; p_{n_1}, p_{n_2}; q_{n_1}, q_{n_2}; x, y)= p_{n_1} p_{n_2}q_{n_1} q_{n_2}\frac{[n_2]_{p_{n_2},q_{n_2}}}{[n_2+1]_{p_{n_2}, q_{n_2}}} \frac{y}{1+y}$
 \item $L_{n_1, n_2}(f_3; p_{n_1}, p_{n_2}; q_{n_1}, q_{n_2}; x, y)= p_{n_1}^3 p_{n_2}q_{n_1}^3 q_{n_2} \frac{[n_1]_{p_{n_1}, q_{n_1}}[n_1-1]_{p_{n_1}, q_{n_1}}}{[n_1+1]_{p_{n_1}, q_{n_1}}^2} \frac{x^2}{(1+x)(p_{n_1}+q_{n_1}x)} + p_{n_1} p_{n_2}q_{n_1} q_{n_2} \frac{[n_1]_{p_{n_1},q_{n_1}}}{[n_1+1]_{p_{n_1}, q_{n_1}}^2} \frac{x}{1+x} + p_{n_1} p_{n_2}^3q_{n_1} q_{n_2}^3 \frac{[n_2]_{p_{n_2}, q_{n_2}}[n_2-1]_{p_{n_2}, q_{n_2}}}{[n_2+1]_{p_{n_2}, q_{n_2}}^2} \frac{y^2}{(1+y)(p_{n_2}+q_{n_2}y)} + p_{n_1} p_{n_2}q_{n_1} q_{n_2}\frac{[n_2]_{p_{n_2},q_{n_2}}}{[n_2+1]_{p_{n_2}, q_{n_2}}} \frac{y}{1+y} $.
\end{enumerate}
\end{lemma}

\begin{proof} Exploiting the proofs for the bivariate operators in \cite{ersan}, the above can be easily established. So we skip the proof.
\end{proof}

Now let the sequences
\begin{equation*}
p=(p_{n_1}), p=(p_{n_2}), q=(q_{n_1}),  q=(q_{n_2})
\end{equation*}
be statistically convergent to unity but not convergent in usual sense, so we can write them for $0<p_{n_1}, q_{n_1}, p_{n_2}, q_{n_2}\leq1$ as
\begin{equation}
st-\lim_{n_1}p_{n_1}= st-\lim_{n_1}q_{n_1}=st-\lim_{n_2}p_{n_2}= st-\lim_{n_2}q_{n_2}=1.
\end{equation}

Now making use of the proof of Theorem (2.2) and conditions (5.2), we establish the statistical convergence of the bivariate operators introduced above.
\begin{theorem} Let $p=(p_{n_1}),~ p=(p_{n_2}),~ q=(q_{n_1})~and ~ q=(q_{n_2})$ be the sequences subject to conditions (2.10) and let $L_{n_1, n_2}$ be the sequence of linear positive operators from $H_{\omega_2}(R_{+}^2)$ into $C_B(R_+)$. Then for each $g\in H_{\omega_2}$,
\begin{equation*}
st-\lim_{n_1, n_2}\|L_{n_1, n_2}(g)-g\| = 0
\end{equation*}
\end{theorem}
\begin{proof}
With the aid of the Lemma (5.2), a proof similar to the proof of the Theorem (2.2) can be easily obtained. So we shall omit the proof.
\end{proof}

\textbf{Rates of convergence of the bivariate operators} \newline
For any $g \in H_{\omega_2}(R_{+}^2)$, the modulus of continuity of the bivariate analogue is defined as:
\begin{eqnarray*}
\tilde{\omega}(g; \delta_1, \delta_2)= \sup_{x_1, x_2\geq0} \bigg\{|g(x_1, y_1) - g(x_2, y_2)|: \left|\frac{x_1}{1+x_1}- \frac{x_2}{1+x_2}\right|\leq \delta_1, \left| \frac{y_1}{1+y_1} - \frac{y_2}{1+y_2}\right|\leq \delta_2, \\
~~ (x_1, y_1),(x_2, y_2)\in H_{\omega_2}(R_{+}^2)\bigg\}
\end{eqnarray*}
For details of this sort of modulus, one is referred to \cite{gal}.

Two chief properties of $\tilde{\omega}(g; \delta_1, \delta_2)$ are
\begin{enumerate}
\item $\tilde{\omega}(g; \delta_1, \delta_2)\rightarrow0 $ ~as $\delta_1\rightarrow0$ ~ and $\delta_2\rightarrow0$ and
\item $|g(x_1, y_1) - g(x_2, y_2)|\leq \tilde{\omega}(g; \delta_1, \delta_2) \left(1+ \frac{\left|\frac{x_1}{1+x_1}- \frac{x_2}{1+x_2}\right|}{\delta_1}\right) \left( 1+\frac{\left| \frac{y_1}{1+y_1} - \frac{y_2}{1+y_2}\right|}{\delta_2}\right)$.
\end{enumerate}

Now in the following theorem we study the rate of statistical convergence of the bivariate operators through modulus of continuity in $H_{\omega_2}$.
\begin{theorem}
Let $p=(p_{n_1}), p=(p_{n_2}), q=(q_{n_1}),  q=(q_{n_2})$ be four sequences obeying conditions of (5.2). Then we have
\begin{equation*}
|L_{n_1, n_2}(f; p_{n_1}, p_{n_2}; q_{n_1}, q_{n_2}; x, y)- f(x, y)|\leq 4 p_{n_1}^2 p_{n_2}^2 q_{n_1}^2 q_{n_2}^2 ~\omega(f; \sqrt{\delta_{n_1}(x)}, \sqrt{\delta_{n_2}(x)}),
\end{equation*}
where
\begin{eqnarray*}
\delta_{n_1}(x) = \frac{x^2}{(1+x)^2} \left( p_{n_1}^2q_{n_1}^2 \frac{(1+x)}{p_{n_1}+q_{n_1}x} \frac{[n_1]_{p_{n_1},q_{n_1}} [n_1 -1]_{p_{n_1},q_{n_1}}}{[n_1+1]_{p_{n_1},q_{n_1}}^2} - 2\frac{[n_1]_{p_{n_1},q_{n_1}}}{[n_1+1]_{p_{n_1},q_{n_1}}} +1\right)\\
+ \frac{x}{1+x} \frac{[n_1]_{p_{n_1},q_{n_1}}}{[n_1+1]_{p_{n_1},q_{n_1}}^2},
\end{eqnarray*}
\begin{eqnarray*}
\delta_{n_2}(y) = \frac{y^2}{(1+y)^2} \left( p_{n_2}^2q_{n_2}^2 \frac{(1+y)}{p_{n_2}+q_{n_1}y} \frac{[n_2]_{p_{n_2},q_{n_2}} [n_2 -1]_{p_{n_2},q_{n_2}}}{[n_1+1]_{p_{n_2},q_{n_2}}^2} - 2\frac{[n_2]_{p_{n_2},q_{n_2}}}{[n_2+1]_{p_{n_2},q_{n_2}}} +1\right)\\
+ \frac{y}{1+y} \frac{[n_2]_{p_{n_2},q_{n_2}}}{[n_2+1]_{p_{n_2},q_{n_2}}^2}.
\end{eqnarray*}

\end{theorem}
\begin{proof}
Using the property of the modulus above, we have
\begin{eqnarray*}
|L_{n_1, n_2}(f; p_{n_1}, p_{n_2}; q_{n_1}, q_{n_2}; x, y)- f(x, y)| 
\leq \omega (f: \delta_{n_1}, \delta_{n_2})\{ L_{n_1, n_2}(f_0; p_{n_1}, p_{n_2}; q_{n_1}, q_{n_2}; x, y)\\
+\frac{1}{\delta_{n_1}} L_{n_1,n_2}(|\frac{t}{1+t}- \frac{x}{1+x}|; p_{n_1}, p_{n_2}; q_{n_1}, q_{n_2}; x, y )\}
\{L_{n_1, n_2}(f_0; p_{n_1}, p_{n_2}; q_{n_1}, q_{n_2}; x, y)\\
+\frac{1}{\delta_{n_2}} L_{n_1,n_2}(|\frac{s}{1+s}- \frac{y}{1+y}|; p_{n_1}, p_{n_2}; q_{n_1}, q_{n_2}; x, y )\}.
\end{eqnarray*}
Applying the Cauchy-Schwarz inequality, we get
\begin{eqnarray*}
L_{n_1,n_2}(|\frac{t}{1+t}- \frac{x}{1+x}|; p_{n_1}, p_{n_2}; q_{n_1}, q_{n_2}; x, y)
\leq\left(L_{n_1,n_2}\left(\left(\frac{t}{1+t}- \frac{x}{1+x}\right)^2; p_{n_1}, p_{n_2}; q_{n_1}, q_{n_2}; x, y \right)\right)^{\frac{1}{2}}\\
 \times \left(L_{n_1, n_2}(f_0; p_{n_1}, p_{n_2}; q_{n_1}, q_{n_2}; x, y)\right)^{\frac{1}{2}}.
\end{eqnarray*}
On substituting this in the above inequality, we get the proof of the theorem.
\end{proof}

Now we shall study the statistical convergence of the bivariate operators using Lipschitz type maximal functions. The Lipschitz type maximal function space on $E\times E \subset \mathbb{R_+ \times \mathbb R_+}$ is defined as follows

\begin{equation}
\tilde{W}_{ \alpha_1, \alpha_2 , E^2} = \bigg\{ f: sup (1+t)^{\alpha_1} (1+s)^{\alpha_2}\tilde{f}_{\alpha_1, \alpha_2} (x, y) \leq M \frac{1}{(1+x)^{\alpha_1}} \frac{1}{(1+y)^{\alpha_2}}; ~~x, y \geq 0, (t, s)\in E^2\bigg\}.
\end{equation}
Where $f$ is a bounded and continuous function on $\mathbb{R_+},~M$ is a positive constant and $0\leq \alpha_1, \alpha_2 \leq1$ and $\tilde{f}_{\alpha_1, \alpha_2} (x, y)$ is defined as follows:
\begin{equation*}
\tilde{f}_{\alpha_1, \alpha_2} (x, y) = \sup_{t, s\geq0} \frac{| f(t, s)- f(x, y)|}{|t-x|^{\alpha_1} |s- y|^{\alpha_2}}.
\end{equation*}

\begin{theorem}
Let $p=(p_{n_1}), p=(p_{n_2}), q=(q_{n_1}),  q=(q_{n_2})$ be four sequences satisfying the conditions of (5.2). Then we have
\begin{eqnarray*}
|L_{n_1, n_2}(f; p_{n_1}, p_{n_2}; q_{n_1}, q_{n_2}; x, y)- f(x, y)|\leq Mp_{n_1}p_{n_2}q_{n_1}q_{n_2} \{ \delta_{n_1}(x)^{\frac{\alpha_1}{2}} \delta_{n_2}(y)^{\frac{\alpha_2}{2}}(p_{n_1}p_{n_2}q_{n_1}q_{n_2})\\
+ \delta_{n_1}(x)^{\frac{\alpha_1}{2}} d(y, E)^{\alpha_2} + \delta_{n_2}(y)^{\frac{alpha_2}{2}} d(x, E)^{\alpha_1} + 2 d(x, E)^{\alpha_1}d(y, E)^{\alpha_2} \},
\end{eqnarray*}
where $0\leq \alpha_1, \alpha_2\leq1$ and $\delta_{n_1}(x)$, $\delta_{n_2}(y)$ are defined as in Theorem (2.11) and $d(x, E) = inf \{ |x-y|: y\in E\}$.
\end{theorem}
\begin{proof}
For $x, y\geq0$ and $(x_1, y_1) \in E\times E$, we can write
\begin{equation*}
|f(t, s) - f(x, y)|\leq |f(t, s) - f(x_1, y_1)| + |f(x_1, y_1) - f(x, y)|.
\end{equation*}
Applying the operator $L_{n_1, n_2}$ to both sides of the above inequality and making use of (2.11), we have
\begin{eqnarray*}
|L_{n_1, n_2}(f; p_{n_1}, p_{n_2}; q_{n_1}, q_{n_2}; x, y)- f(x, y)|\leq L_{n_1,n_2} ( |f(t, s)- f(x_1, y_1)|; p_{n_1}, p_{n_2}; q_{n_1}, q_{n_2}; x, y)\\
+ |f(x_1, y_1) - f(x, y)| L_{n_1, n_2}(f_0; p_{n_1}, p_{n_2}; q_{n_1}, q_{n_2}; x, y)\\
\leq ML_{n_1,n_2}\left( \big|\frac{y_2}{1+y_2} - \frac{x_1}{1+x_1}\big|^{\alpha_1} \big|\frac{x_2}{1+x_2}- \frac{y_1}{1+y_1}\big|^{\alpha_2}; p_{n_1}, p_{n_2}; q_{n_1}, q_{n_2}; x, y)\right)\\
+M \big|\frac{x}{1+x}-\frac{x_1}{1+x_1}\big|^{\alpha_1} \big|\frac{y}{1+y} - \frac{y_1}{1+y_1}\big|^{\alpha_2}L_{n_1, n_2}(f_0; p_{n_1}, p_{n_2}; q_{n_1}, q_{n_2}; x, y).
\end{eqnarray*}
Now for $0\leq p\leq1$, using $(a+b)^p\leq a^p +b^q$, we can write
\begin{equation*}
\big|\frac{y_2}{1+y_2} - \frac{x_1}{1+x_1}\big|^{\alpha_1}\leq \big|\frac{y_2}{1+y_2} - \frac{x}{1+x}\big|^{\alpha_1}+ \big|\frac{x}{1+x}-\frac{x_1}{1+x_1}\big|^{\alpha_1}
\end{equation*}
and
\begin{equation*}
\big|\frac{x_2}{1+x_2} - \frac{y_1}{1+y_1}\big|^{\alpha_2}\leq \big|\frac{x_2}{1+x_2} - \frac{y}{1+y}\big|^{\alpha_2}+ \big|\frac{y}{1+y}-\frac{y_1}{1+y_1}\big|^{\alpha_2}.
\end{equation*}
Using these inequalities in the above, we get
\begin{eqnarray*}
|L_{n_1, n_2}(f; p_{n_1}, p_{n_2}; q_{n_1}, q_{n_2}; x, y)- f(x, y)|
\leq L_{n_1,n_2} ( \big|\frac{y_2}{1+y_2} - \frac{x}{1+x}\big|^{\alpha_1}
\big|\frac{x_2}{1+x_2}-\frac{y}{1+y}\big|^{\alpha_2}\\
; p_{n_1}, p_{n_2}; q_{n_1}, q_{n_2}; x, y)
+\big|\frac{y}{1+y} - \frac{y_1}{1+y_1}\big|^{\alpha_2}
L_{n_1, n_2}(\big|\frac{y_2}{1+y_2} - \frac{x}{1+x}\big|^{\alpha_1}
; p_{n_1},p_{n_2}; q_{n_1}, q_{n_2}; x, y)\\
+ \big|\frac{x}{1+x}-\frac{x_1}{1+x_1}\big|^{\alpha_1}
 L_{n_1, n_2}(\big|\frac{x_2}{1+x_2} - \frac{y}{1+y}\big|^{\alpha_2}; p_{n_1},p_{n_2}; q_{n_1}, q_{n_2}; x, y)
+\big|\frac{x}{1+x}-\frac{x_1}{1+x_1}\big|^{\alpha_1}\\
\times\big|\frac{y}{1+y} - \frac{y_1}{1+y_1}\big|^{\alpha_2}L_{n_1, n_2}(f_0; p_{n_1}, p_{n_2}; q_{n_1}, q_{n_2}; x, y).
\end{eqnarray*}
Now using the H\"{o}lder's inequality for $p_1 = \frac{2}{\alpha_1}, p_2 = \frac{2}{\alpha_2}, q_1 = \frac{2}{2-\alpha_1} q_2 = \frac{2}{2-\alpha_2}$, we get
\begin{eqnarray*}
L_{n_1,n_2} ( \big|\frac{y_2}{1+y_2} - \frac{x}{1+x}\big|^{\alpha_1}
\big|\frac{x_2}{1+x_2} - \frac{y}{1+y}\big|^{\alpha_2}; p_{n_1}, p_{n_2}; q_{n_1}, q_{n_2}; x, y)
 = L_{n_1, n_2}(\big|\frac{y_2}{1+y_2} - \frac{x}{1+x}\big|^{\alpha_1}\\
 ; p_{n_1},p_{n_2}; q_{n_1}, q_{n_2}; x, y)L_{n_1, n_2}(\big|\frac{x_2}{1+x_2} - \frac{y}{1+y}\big|^{\alpha_2}; p_{n_1},p_{n_2}; q_{n_1}, q_{n_2}; x, y)
 \leq (L_{n_1, n_2}(\frac{y_2}{1+y_2}\\
  - \frac{x}{1+x})^2; p_{n_1},p_{n_2}; q_{n_1}, q_{n_2}; x, y)^{\frac{\alpha_1}{2}} (L_{n_1, n_2}(f_0; p_{n_1},p_{n_2}; q_{n_1}, q_{n_2}; x, y ))^{\frac{2-\alpha_1}{2}}
 (L_{n_1, n_2}(\frac{x_2}{1+x_2}\\
  - \frac{y}{1+y})^2; p_{n_1},p_{n_2}; q_{n_1}, q_{n_2}; x, y)^{\frac{\alpha_2}{2}} (L_{n_1, n_2}(f_0; p_{n_1},p_{n_2}; q_{n_1}, q_{n_2}; x, y ))^{\frac{2-\alpha_2}{2}}.
\end{eqnarray*}
This consequently gives the desired result. Therefore the proof is complete.
\end{proof}

\begin{remark}
For $E = [0, \infty)$, we see that $d(x, E) = 0$ and $d(y, E) = 0$ , so that we have
$|L_{n_1, n_2}(f; p_{n_1}, p_{n_2}; q_{n_1}, q_{n_2}; x, y)- f(x, y)|\leq M(p_{n_1}p_{n_2}q_{n_1}q_{n_2})^{4-\frac{\alpha_1 + \alpha_2}{2}} \delta_{n_1}(x)^{\frac{\alpha_1}{2}} \delta_{n_2}(y)^{\frac{\alpha_2}{2}}$.
\end{remark}

\begin{remark}
By means of (2.10), it can be easily seen that $st- \lim_{n_1}\delta_{n_1} = 0$ and $st-\lim_{n_2} \delta_{n_2} = 0$. So we can estimate the order of statistical approximation of our bivariate operators by means of Lipschitz type maximal functions using this result.
\end{remark}
Also as
\begin{equation*}
\sup_{x\geq0} \delta_{n_1}(x)\leq \frac{p_{n_1}^{2n_1}q_{n_1}^{2n_1}}{[n_1+1]_{p_{n_1}, q_{n_1}}^2}
\end{equation*}
and
\begin{equation*}
p_{n_1}^{n_1}q_{n_1}^{n_1}(n_1 +1)\leq \left( \frac{1}{p_{n_1}^{n_1}q_{n_1}^{n_1}} + ... + \frac{1}{p_{n_1}q_{n_1}} +1 \right)p_{n_1}^{n_1}q_{n_1}^{n_1}
\end{equation*}
So for $0\leq p_{n_1}, q_{n_1}\leq1$, we get
\begin{equation*}
\frac{p_{n_1}^{2n_1}q_{n_1}^{2n_1}}{[n_1+1]_{p_{n_1}, q_{n_1}}^2}\leq \frac{1}{(n_1 +1)^2}.
\end{equation*}
In a similar fashon we can obtain it for $\delta_{n_2}(y)$. So we have the following concluding remark.
\begin{remark} This chapter has two main features:
\begin{enumerate}
\item $\delta_{n_1}$ and $\delta_{n_2}$ approach to zero in statistical sense however they may not tend to zero in the usual sense.
\item In our case $\delta_{n_1}$ and $\delta_{n_2}$ approach to zero faster than that of the classical BBH operators.
\end{enumerate}
\end{remark}

\bibliographystyle{amsplain}

\begin{thebibliography}{99}
\bibitem{gal} G. A. Anastassiou, S. G. Gal, Approximation theory : \emph{Modulii of continuity and global smoothness preservation}, Birkhauser, Boston, 2000.
\bibitem{aral1} A. Aral and O. Doًru, \emph{Bleimann Butzer and Hahn operators
based on $q$-integers}, J. Inequal. Appl., (2007) 1-12. Art. ID 79410.

\bibitem{brns} G. Bleimann, P.L. Butzer and L. Hahn, \emph{A Bernstein-type
operator approximating continuous functions on the semi-axis}, Indag. Math.,
42 (1980) 255-262.

\bibitem{n1} O. Doًru, "O\emph{n Bleimann, Butzer and Hahn type generalization of
Balلzs operators,}" Stud. Univ. Babe‏-Bolyai. Math., 47(4) (2002) 37-45.
\bibitem{erkus} E. Erku\c{s}, O. Duman, \emph{A-Statistical extension of the Korovkin type approximation theorem}, Proc. Indian Acad. Sci. Math. Sci. 115(4) (2003) 499-507.
\bibitem{ersan} S. Ersan, \emph{Approximation properties of bivariate generalization of Bleimann, Butzer and Hahn operators based on the q-integers}, in : Proc. of the 12th WSEAS Int. Conference on Applied Mathematics, Cairo, Egypt, 2007, pp. 122-127.
\bibitem{fast} H. Fast, \emph{Sur la convergence statistique}, Colloq. Math. 2(1951) 241-244.
\bibitem{butz} A.D. Gadjiev and ض. Cakar, \emph{On uniform approximation by
Bleimann, Butzer and Hahn operators on all positive semi-axis}, Trans. Acad.
Sci. Azerb.Ser. Phys. Tech. Math. Sci., 19 (1999) 21-26.

\bibitem{mah} M.N. Hounkonnou, J. Désiré and B. Kyemba, \emph{$\mathcal{R}(p,q)$%
-calculus: differentiation and integration}, SUT Jour. Math., 49(2) (2013)
145-167.

\bibitem{lenz} B. Lenze, \emph{Bernstein-Baskakov-Kantorovich operators and
Lipschitz-type maximal functions}, in: Colloq. Math. Soc. Janos Bolyai, 58,
Approx. Th., (1990) 469-496.

\bibitem{lups} A. Lupa‏, \emph{A $q$-analogue of the Bernstein operator},
University of Cluj- Napoca, Seminar on Numerical and Statistical Calculus, 9
(1987) 85-92.
\bibitem{mur7} M. Mursaleen, K. J. Ansari and A. Khan, \emph{On $(p,q)$-analogue
of Bernstein operators}, Appl. Math. Comput.,
266(2015), 874-882.

\bibitem{mur8} M. Mursaleen, K. J. Ansari and A. Khan, \emph{Some approximation
results by $(p,q)$-analogue of Bernstein-Stancu operators}, Appl. Math.
Comput., 264 (2015), 392-402.

\bibitem{mur3} M. Mursaleen and A. Khan, \emph{Generalized $q$-Bernstein-Schurer
operators and some approximation theorems}, J. Funct. Spaces Appl., Volume
2013.
\bibitem {mur18} M. Mursaleen, Md. Nasiruzzaman, Asif Khan and Khursheed Ansari, \emph{Some approximation results on Bleimann-Butzer-Hahn operators defined by (p, q)-integrs} (accepted).
 \bibitem {niv} I. Niven, H. S. Zuckerman, H. Montgomery, \emph{An introduction to the theory of numbers}, 5th edition, Wiley, New York, 1991.

\bibitem{philip} G.M. Phillips, \emph{Bernstein polynomials based on the $q$%
-integers}, The heritage of P.L.Chebyshev: Ann. Numer. Math. 4 (1997) 511-518.

\bibitem{sad} P. N. Sadjang, \emph{On the fundamental theorem of $(p,q)$-calculus
and some $(p,q)$-Taylor formulas}, arXiv:1309.3934 [math.QA].

\bibitem{vivek} V. Sahai and S. Yadav,\emph{ Representations of two parameter
quantum algebras and $p$,$q$-special functions}, J. Math. Anal. Appl. 335
(2007) 268-279.

\bibitem{n2} D. D. Stancu, \emph{"Approximation of functions by a new class of
linear polynomial operators,"} Rev. Roumaine Math. Pures Appl., 13 (1968)
1173-1194.

\bibitem{vp} K. Victor and C. Pokman,\emph{ \textit{Quantum Calculus}},
Springer-Verlag, New York Berlin Heidelberg, 2002.




\end{thebibliography}

\end{document}